\documentclass[11pt,letterpaper]{article}

\usepackage[margin=1in]{geometry}
\usepackage[T1]{fontenc}
\usepackage{microtype}

\usepackage{mathtools,amssymb,amsthm}
\usepackage{newtxtext,newtxmath}

\newtheorem{theorem}{Theorem}[section]
\newtheorem{proposition}[theorem]{Proposition}

\newtheorem{corollary}[theorem]{Corollary}

\theoremstyle{definition}

\theoremstyle{remark}

\usepackage{graphicx}
\usepackage{booktabs}
\usepackage{tabularx}
\usepackage{enumitem}
\usepackage{xcolor}
\usepackage{svg}
\usepackage{tikz}
\usepackage[labelsep=period]{caption}
\usetikzlibrary{intersections, calc}

\setlist{
  itemsep=0.25em,
  topsep=0.4em
}

\usepackage[numbers,sort&compress]{natbib}

\usepackage{xcolor}
\definecolor{linkblue}{RGB}{25,65,120}

\usepackage[
  pdfusetitle,
  colorlinks=true,
  linkcolor=linkblue,
  citecolor=linkblue,
  urlcolor=linkblue
]{hyperref}

\title{\textbf{In Search of Melchior's Ordinary Points}}

\author{
  Jonathan Lenchner\\
  \small IBM Research\\
  \small Yorktown Heights\\
  \small \texttt{jon.lenchner@gmail.com}
  \and
  Rik Sengupta\\
  \small IBM Research\\
  \small Cambridge, MA\\
  \small \texttt{rik@ibm.com}
}

\date{August 2026}

\begin{document}

\maketitle

\begin{abstract}
In 1893, James Joseph Sylvester posed the following problem: given $n$ points in the plane, not all collinear, must there be a line determined by two of the points that does not pass through any of the other points? In 1940, Eberhard Melchior studied the equivalent dual problem in the projective plane: given a set of $n$ lines in the (real) projective plane, not all passing through a common point, must there be a point where exactly two of the lines intersect? Such a point of intersection is called an \emph{ordinary point}. Via a clever double-counting argument, Melchior found that in fact there must be at least three such points. Given the many simple ``visual'' proofs of what is today known as the Sylvester--Gallai Theorem---the theorem that states there must be at least one ordinary point---a natural question is whether there is a simple visual proof that recovers all three of Melchior’s ordinary points. This paper provides such a proof. 
\end{abstract}


\medskip

\noindent\textit{Preprint version of an article in \emph{The American Mathematical Monthly}, pp 1-8, July 2026 \cite{mainarticle}.}

\section{Introduction.}

In 1893, James Joseph Sylvester posed the following problem \cite{sylv93}: given a set of $n$ points in the Euclidean plane, not all of which are collinear, must there be a line determined by (i.e., passing through) two of the points that does not pass through any additional point in the set? The first recognized solution to Sylvester's problem, answering the question in the affirmative, was due to Tibor Gallai in 1944 \cite{gallai44}. However, after Gallai's solution was published, it was realized that a solution to Sylvester's problem had appeared earlier, in 1940, in the work of Eberhard Melchior \cite{melch41}. Melchior had studied the dual problem: given a set of $n$ lines in the real projective plane, not all intersecting in a common point, must there be a point of intersection of exactly two of the lines?  By projective duality (see the discussion at the beginning of Section \ref{sec:prelims}), the two problems are equivalent. 
In this paper we confine our attention entirely to the dual setting.
In this setting, a point of intersection of exactly two lines is called an \textit{ordinary point}.
The affirmative conclusion to the problem, either in the primal or dual setting, is usually referred to as the Sylvester--Gallai Theorem. Today, many proofs of the Sylvester--Gallai Theorem are known \cite{pfs_fr_the_book:2018, borwein90, bmp2006}. However, a novel feature of Melchior's proof is that, via a clever double-counting argument, it actually establishes that there must be at least \textit{three} ordinary points. 

Given the many simple ``visual'' (in other words, essentially algorithmic)
proofs of the Sylvester--Gallai Theorem that are now known, a natural question to ask is whether there is an explicit constructive argument that, given a set of $n$ lines, directly finds three ordinary points, without appealing implicitly to Melchior's double-counting argument and then testing all intersection points. This paper provides such an argument.

\section{History and Related Work.}\label{sec:history}

After Sylvester posed his problem in 1893 \cite{sylv93}, 
Paul Erd\H{o}s re-posed the problem in the \textit{American Mathematical Monthly} in 1943 \cite{erdos43}, and it was answered in the affirmative by Tibor Gallai in an issue of the \textit{Monthly} that appeared the next year \cite{gallai44}.
It is not entirely clear when or by whom Melchior's 1940 proof \cite{melch41} was first noticed. 
Melchior's double-counting argument using Euler's relation is now 
commonly encountered in the study of line arrangements in computational geometry \cite{cg-aa:2008}. In the decades following Melchior, a substantial body of work has accumulated on the minimum possible number of ordinary points (respectively, lines) as a function of the number, $n$, of lines (respectively, points), in an arrangement of not-all-concurrent lines (respectively, a set of not-all-collinear points). 
In 1958, Leroy Kelly and William Moser showed that there are at least $\frac{3n}{7}$ ordinary points \cite{km58}. Then, in 1993, Joseph Csima and Eric Sawyer showed that for $n > 7$, there are at least $\frac{6n}{13}$ ordinary points \cite{cs93}. In 2013, Ben Green and Terence Tao showed that for sufficiently large $n$, there must be at least $\frac{n}{2}$ ordinary points \cite{gt13}. 
All three results use Melchior's argument as a starting point.


In this paper, we analyze a relatively recent and simple proof of the Sylvester--Gallai Theorem \cite{lench08}, and show that it can be adapted to extract all three of Melchior's ordinary points (the minimum number of ordinary points, independent of the number of lines), without using Euler's relation or a counting argument.

\section{Preliminaries.} \label{sec:prelims}
Since Melchior's argument takes place in the real projective plane, we start with a quick overview of the geometry of this plane. 
In the real projective plane, $\mathbb{RP}^2$, every two distinct lines intersect in a unique point and, just like in the Euclidean plane, every two distinct points determine a unique line. One way to construct a model of the projective plane is to start with the Euclidean plane, add ``points at infinity'' corresponding to the possible slopes of lines, and then collect all the points at infinity together at a new ``line at infinity''. A line of slope $\sigma$, with $\sigma \in \mathbb{R} \cup \{\infty\}$, then passes through the point at infinity associated with slope $\sigma$. The \emph{principle of duality} in the real projective plane says that, for any theorem in this geometry that talks solely about incidences of points and lines, there exists a corresponding dual theorem derived by interchanging the roles of points and lines. This means that if one 
has a statement about points, lines, and their incidence relationships, one can swap ``point'' for ``line'' and vice versa, and the resulting statement will also be a valid theorem. Therefore, we can start with the $n$ points from Sylvester's problem, and transfer them to the projective plane and dualize the points to lines, to get an equivalent dual problem. In addition to the principle of duality, we later utilize an important family of mappings of the real projective plane known as projective transformations. These are the incidence-preserving bijections of the projective plane onto itself. 
There is always a projective transformation that takes an arbitrary given line to the line at infinity. 
For a more thorough treatment of projective geometry, we refer the reader to \cite{bennett:1995}.

Given a finite set of lines in the projective plane, not all passing through a common point, one can consider the resulting partition of the plane into vertices, edges, and faces induced by the lines---what is referred to as the \textit{line arrangement} induced by the set of lines. If we let $\mathcal{V}$, $\mathcal{E}$, and $\mathcal{F}$ denote the set of vertices, edges, and faces in this partition respectively, and define $V := |\mathcal{V}|$, $E := |\mathcal{E}|$, and $F := |\mathcal{F}|$, then Euler's relation in the projective plane \cite[Prop. 5.3]{felsner2004} 
says that we have:
\begin{equation} \label{euler}
    V - E + F = 1.
\end{equation}

Given an arrangement of lines in the real projective plane, and a subset $\mathcal{S}$ of these lines, by a \textit{region} associated with $\mathcal{S}$, we mean a closed, connected subset of $\mathbb{RP}^2$ whose boundary consists of segments lying in the lines of $\mathcal{S}$. A \emph{triangular region} is a region bounded by the line segments of three lines. For example, in Figure \ref{fig:3_line_arrangement} below, the three lines $\ell_1, \ell_2$, and $\ell_3$ partition the projective plane into four triangular regions, respectively denoted $T_1, T_2, T_3$, and $T_4$. 
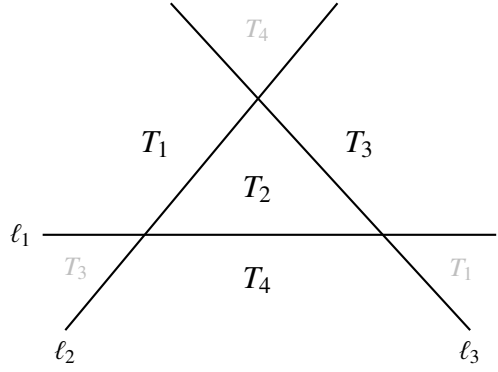
\begin{figure*}[h]
\begin{center}
\begin{tikzpicture}[scale=1.5, every node/.style={font=\small}]
  \coordinate (v0) at (0,1.2);
  \coordinate (v1) at (-1,0);
  \coordinate (v3) at (1.1,0);

  \draw[thick] (2.1,0) -- (-1.9,0) node[left] {$\ell_1$};

  \draw[thick] (v0) -- (v1);
  \draw[thick] (v0) -- (v3);

  \draw[thick] (v1) -- ($(v0)!1.7!(v1)$) node[below] {$\ell_2$};
  \draw[thick] (v3) -- ($(v0)!1.7!(v3)$) node[below] {$\ell_3$};
  \draw[thick] (v0) -- ($(v1)!1.7!(v0)$);
  \draw[thick] (v0) -- ($(v3)!1.7!(v0)$);



  \node at (-0.9,0.8) {\large $T_1$};
  \node at (0,0.4) {\large $T_2$};
  \node at (0.9,0.8) {\large $T_3$};
  \node at (0,-0.4) {\large $T_4$};

  \node at (0,1.8) {\textcolor{lightgray}{$T_4$}};
  \node at (-1.6,-0.3) {\textcolor{lightgray}{$T_3$}};
  \node at (1.8,-0.3) {\textcolor{lightgray}{$T_1$}};
  
\end{tikzpicture}
\end{center}
\caption{There exist three lines, $\ell_1, \ell_2$ and $\ell_3$, of the arrangement which do not intersect in a common point. These lines partition $\mathbb{RP}^2$ into four triangular regions.}
\label{fig:3_line_arrangement}
\end{figure*}
The \emph{defining vertices} of such a region are the pairwise intersection points of the three lines; the \emph{defining edges} are the segments of the lines forming the boundary of the region. We are depicting just three of possibly many lines in the arrangement, so there can be many lines passing into each of these regions, and the regions will not, in general, simply be faces of the arrangement. Furthermore, if none of the edges of the faces  contained in a region pass through, or lie on, the line at infinity, we call the region \emph{finite}. In an arrangement of lines where not all lines pass through a common point, we can pick any particular triangular region, apply a projective transformation taking a line not intersecting the region (including its boundary) to the line at infinity, and thereby make the region finite.

\section{Melchior's Proof.}\label{sec:melchior}

Given a set of $n$ lines in the projective plane, not all passing through a common point (and hence $n \geq 3$), Melchior considered the induced line arrangement. 
To understand his argument, we introduce the following notation:
\begin{eqnarray*}
    N_k &=& \textrm{number of vertices where exactly $k$ lines cross,} \\
    M_s &=& \textrm{number of faces bounded by exactly $s$ edges}.
\end{eqnarray*}
We now make a series of observations. First, note that:
\begin{equation} \label{r1}
   \sum_{k=2}^n N_k = V.
\end{equation}
Next, consider a face, $f$, with exactly two edges, i.e., a face counted in $M_2$, and say the lines giving rise to these edges are $\ell_1$ and $\ell_2$. The vertex $v = \ell_1 \cap \ell_2$ is then a defining vertex of $f$. Since not all lines in the arrangement pass through the same point, this means that at least one other line of the arrangement intersects $\ell_1$ and $\ell_2$ at  vertices other than $v$; consider a closest such vertex to $v$ along $\ell_1$ or $\ell_2$ and call this vertex $v'$. Then, $v'$ would be a defining vertex of $f$, and some line (distinct from $\ell_1$ and $\ell_2$) passing through $v'$ would give rise to an edge of $f$ that is different from the edges of $f$ along $\ell_1$ and $\ell_2$. This contradicts $f$ having just two edges. 
A two-edged face therefore cannot exist, and so $M_2 = 0$. Hence:
\begin{equation} \label{r2}
   \sum_{s=3}^n M_s = F. 
\end{equation}
Next, recall that the \emph{degree} of a vertex $v$, denoted $\deg(v)$, is the number of edges incident to it; each line passing through a vertex $v$, in general, contributes $2$ to its degree. The only way the degree of a vertex could have a contribution from only one edge along a given line is if \emph{all} lines in the collection passed through this same vertex, which, again, is prohibited by assumption. It follows that all lines contain at least two vertices. We will now compute the sum of the degrees of all vertices in two different ways. On the one hand, every vertex counted in $N_k$ has exactly $k$ lines passing through it, and so it has degree $2k$. Further, every edge contributes twice to $\sum_{v \in \mathcal{V}}\textrm{deg}(v)$, once at each of its endpoints. Hence, we have:
\begin{equation} \label{r3}
    2E = \sum_{v \in \mathcal{V}}\textrm{deg}(v) = \sum_{k=2}^n 2kN_k.
\end{equation}
We can analogously count the number of edges bounding each face. For any face $f$, define $\textrm{deg}(f)$ as the number of edges bounding $f$. Any face counted in $M_s$ has exactly $s$ edges bounding it. Since every edge is shared by exactly two faces, we obtain:
\begin{equation} \label{r4}
    2E = \sum_{f \in \mathcal{F}}\textrm{deg}(f) = \sum_{s=3}^n sM_s.
\end{equation}
Melchior then considered the quantity $Y$, defined as:
\begin{equation} \label{melch-y}
    Y := \sum_{k=2}^n(3-k)N_k + \sum_{s=3}^n(3-s)M_s.
\end{equation}
Breaking the right hand side of (\ref{melch-y}) into four pieces, and substituting in (\ref{r1})-(\ref{r4}), we obtain:
\begin{eqnarray}
        Y &=& 3\sum_{k=2}^nN_k - \sum_{k=2}^nkN_k + 3\sum_{s=3}^nM_s - \sum_{s=3}^nsM_s \notag \\
        &=& 3V - E + 3F - 2E \notag\\ 
        &=& 3(V - E + F) \notag\\
        &=& 3. \label{p1}
\end{eqnarray}
The last equality follows from Euler's relation (\ref{euler}).

But now, we can go back and rewrite Melchior's expression (\ref{melch-y}) as follows:
\begin{align} \label{melch-y-new}
    Y &= \sum_{k=2}^n(3-k)N_k + \sum_{s=3}^n(3-s)M_s \nonumber \\
    3 &= N_2 + \sum_{k=3}^n(3-k)N_k + \sum_{s=3}^n(3-s)M_s.
\end{align}
Since all terms in the two sums on the right hand side of (\ref{melch-y-new}) are nonpositive, we conclude that $N_2$, the number of ordinary points, is at least $3$. \qed

%

\section{Finding Melchior's Three Ordinary Points.} \label{sec:one_op}

In this section, we present our main result. A sketch of the argument is as follows: given a collection of $n$ lines, we start by considering a triangular region determined by some three of the lines, with a fourth line passing through a defining vertex (formed by the intersection of two of the three lines) into the interior of the region, and intersecting the third line in a vertex $v$.
We will inductively show that either $v$ is an ordinary point, or there is a triangular region properly contained in the initial triangular region, with a line passing through one of its defining vertices, creating an ordinary point at its intersection with the opposite edge of that triangular region. From this it will follow that the original triangular region contains an ordinary point that is not one of its defining vertices. Looking at vertices on either side of this first ordinary point along a given line, we will be able to recover two additional ordinary points using the same inductive technique.

Start with any three lines from the arrangement, say $\ell_1, \ell_2$, and $\ell_3$, which do not meet in a common point. Three such lines must exist by the assumption that not all lines meet in a common point. 
These three lines partition the projective plane into four triangular regions; we have labeled these regions $T_1, T_2, T_3$, and $T_4$ in Figure \ref{fig:3_line_arrangement}. Note that three of the regions ``wrap around the line at infinity'' as we have indicated. 

If the three points of intersection of the three lines are all ordinary, we have found our three ordinary points, so assume one of the intersection points, say $v_1 = \ell_2 \cap \ell_3$, is not ordinary. There is then an additional line, say $\ell_4$, passing through $v_1$, which passes through the interior of two of the four triangular regions. 
Consider either of these two triangular regions and apply a projective transformation that maps any chosen line disjoint from the triangular region to the line at infinity, thereby 
making the region finite. We will henceforth focus just on this finite triangular region, which we take without loss of generality to be the region $T_2$ in Figure \ref{fig:3_line_arrangement}. The arrangement with the addition of the line $\ell_4$ is depicted in Figure \ref{fig:3_line_arrangement_step1}.

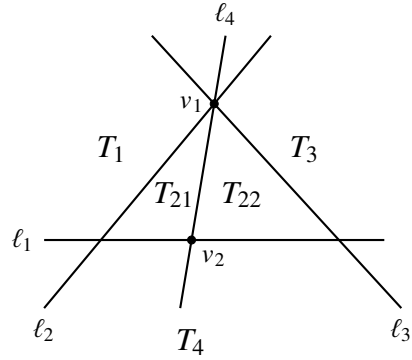
\begin{figure*}[ht]
\begin{center}
\begin{tikzpicture}[scale=1.5, every node/.style={font=\small}]
  \coordinate (v0) at (0,1.2);
  \coordinate (v1) at (-1,0);
  \coordinate (v2) at (-0.2,0);
  \coordinate (v3) at (1.1,0);

  \draw[thick] (1.5,0) -- (-1.5,0) node[left] {$\ell_1$};

  \draw[thick] (v0) -- (v1);
  \draw[thick] (v0) -- (v2);
  \draw[thick] (v0) -- (v3);

  \draw[thick] (v1) -- ($(v0)!1.5!(v1)$) node[below] {$\ell_2$};
  \draw[thick] (v3) -- ($(v0)!1.5!(v3)$) node[below] {$\ell_3$};
  \draw[thick] (v2) -- ($(v0)!1.5!(v2)$);
  \draw[thick] (v0) -- ($(v1)!1.5!(v0)$);
  \draw[thick] (v0) -- ($(v2)!1.5!(v0)$) node[above] {$\ell_4$};
  \draw[thick] (v0) -- ($(v3)!1.5!(v0)$);

  \filldraw[black] (v0) circle (1pt) node[left] {$v_1$};
  \filldraw[black] (v2) circle (1pt) node[below right] {$v_2$};


  \node at (-0.35,0.4) {\large $T_{21}$};
  \node at (0.25,0.4) {\large $T_{22}$};
  \node at (-0.2,-0.9) {\large $T_4$};
  \node at (-0.9,0.8) {\large $T_1$};
  \node at (0.8,0.8) {\large $T_3$};

\end{tikzpicture}
\end{center}
\caption{The additional line, $\ell_4$, through $v_1$,  passes into the finite triangular region $T_2$, subdividing it into smaller triangular regions $T_{21}$ and $T_{22}$.}
\label{fig:3_line_arrangement_step1} 
\end{figure*}

%
%
We now prove by strong induction on $k \geq 0$ that for any triangular region $T$, and line $\ell$ passing through one of the defining vertices of $T$ into the interior of $T$ and meeting the opposite edge of $T$ at a vertex $v$, along with $k$ additional lines passing through the interior of $T$, either (i) $v$ is ordinary, or (ii) $T$ properly contains some other triangular region $T'$ with a line $\ell'$ passing through one of its defining vertices, which meets the opposite edge of $T'$ at an ordinary vertex $v'$.

In short, we have a scenario of the type depicted in Figure \ref{fig:3_line_arrangement_step1}, but with the addition of $k$ undrawn lines that pass through the interior of $T_{21} \cup T_{22}$. 
If $k = 0$, there are no additional lines, and so the vertex $v$ (i.e., $v_2$ in Figure \ref{fig:3_line_arrangement_step1}) 
is ordinary, and we are done. Thus, assume our assertion is true in the case of any number of additional lines up to $k$, and let us prove it for $k + 1$ additional lines. From Figure \ref{fig:3_line_arrangement_step1}, either the vertex $v_2$ is ordinary, in which case we are done, or there is an additional line through $v_2$ passing into the interior of either $T_{21}$ or $T_{22}$. The two cases are symmetric, so assume the additional line passes into $T_{22}$, intersecting the line $\ell_3$ at a vertex $v_3$. Since there were $k + 1$ lines in addition to $\ell_4$ passing into the interior of $T_2 = T_{21} \cup T_{22}$, there are at most $k$ lines in addition to the new line passing into the interior of $T_{22}$, and so we can use the induction hypothesis to conclude that either $v_3$ is ordinary, or $T_{22}$ contains an ordinary point arising from a nested triangular region and additional line. The same conclusion therefore applies to $T_2$, and we have therefore established the following:

\begin{proposition}
\label{prop_special_ord_point}
Let $T$ be a triangular region in a finite line arrangement, and suppose that a line $\ell$ of the arrangement passes through a defining vertex of $T$, into the interior of $T$, and intersects the opposite edge of $T$ at a vertex $v$. Then, either $v$ is ordinary, or $T$ properly contains another triangular region, $T'$, with an additional line $\ell'$ of the arrangement passing through a defining vertex of $T'$, into the interior of $T'$, and intersecting the opposite edge of $T'$ at an ordinary point.
\end{proposition} 

Since the ordinary point found in Proposition \ref{prop_special_ord_point} is never one of the defining vertices of the triangular region $T$, we have:

\begin{corollary} \label{prop_adj-triangles}
Let $T$ be a triangular region in a finite line arrangement and suppose a line of the arrangement passes through a defining vertex of $T$ into its interior. Then, $T$ must contain an ordinary point that is not one of its three defining vertices.
\end{corollary}

The next proposition will allow us to find the three ordinary points guaranteed by Melchior. 

\begin{proposition} \label{prop-final}
Let $T$ be a triangular region in a finite line arrangement and suppose a line of the arrangement passes through a defining vertex of $T$ into its interior, intersecting the opposite edge of $T$ at the vertex $v$. If $v$ is an ordinary point, then the arrangement contains at least three ordinary points.
\end{proposition}

\begin{proof}
To establish the Proposition, let us consider Figure \ref{fig:three-vertices-more-triangles}, where we have labeled more vertices and regions than in the preceding figures. The vertex that plays the role of the vertex $v$ in the statement of the Proposition is here labeled $v_2$.
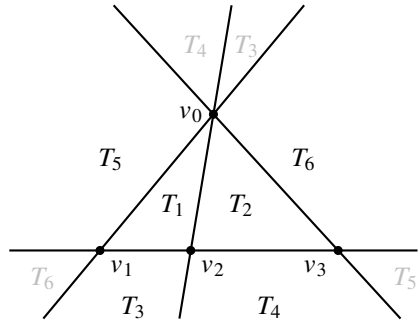
\begin{figure*}[ht]
\begin{center}
\begin{tikzpicture}[scale=1.5, every node/.style={font=\small}]
  \coordinate (v0) at (0,1.2);
  \coordinate (v1) at (-1,0);
  \coordinate (v2) at (-0.2,0);
  \coordinate (v3) at (1.1,0);

  \draw[thick] (1.8,0) -- (-1.8,0);

  \draw[thick] (v0) -- (v1);
  \draw[thick] (v0) -- (v2);
  \draw[thick] (v0) -- (v3);

  \draw[thick] (v1) -- ($(v0)!1.5!(v1)$);
  \draw[thick] (v3) -- ($(v0)!1.5!(v3)$);
  \draw[thick] (v2) -- ($(v0)!1.5!(v2)$);
  \draw[thick] (v0) -- ($(v1)!1.8!(v0)$);
  \draw[thick] (v0) -- ($(v2)!1.8!(v0)$);
  \draw[thick] (v0) -- ($(v3)!1.8!(v0)$);

  \filldraw[black] (v0) circle (1pt) node[left] {$v_0$};
  \filldraw[black] (v1) circle (1pt) node[below right] {$v_1$};
  \filldraw[black] (v2) circle (1pt) node[below right] {$v_2$};
  \filldraw[black] (v3) circle (1pt) node[below left] {$v_3$};


  \node at (-0.35,0.4) {\textcolor{black}{$T_1$}};
  \node at (0.25,0.4) {\textcolor{black}{$T_2$}};
  \node at (-0.7,-0.5) {\textcolor{black}{$T_3$}};
  \node at (0.5,-0.5) {\textcolor{black}{$T_4$}};
  \node at (-0.9,0.8) {\textcolor{black}{$T_5$}};
  \node at (0.8,0.8) {\textcolor{black}{$T_6$}};
  
  \node at (-0.15,1.8) {\textcolor{lightgray}{$T_4$}};
  \node at (0.3,1.8) {\textcolor{lightgray}{$T_3$}};
  \node at (-1.5,-0.25) {\textcolor{lightgray}{$T_6$}};
  \node at (1.7,-0.25) {\textcolor{lightgray}{$T_5$}};

\end{tikzpicture}
\end{center}
\caption{
The analog of Figure \ref{fig:3_line_arrangement_step1} but with a few more labeled triangular regions and vertices.}
\label{fig:three-vertices-more-triangles}
\end{figure*}

By assumption, $v_2$ is ordinary. Now, suppose $v_1$ is \textit{not} ordinary. Then there is an additional line through $v_1$ that passes through either the pair of triangular regions $T_1$ and $T_6$, or the pair of triangular regions $T_3$ and $T_5$. Observe that besides the vertex $v_1$, $T_1$ and $T_6$ share just one point in common: the vertex $v_0$. The same is also true for the triangular regions $T_3$ and $T_5$; they too just share the common points $v_0$ and $v_1$. If the additional line through $v_1$ passes into $T_1$ and $T_6$, then it splits each of $T_1$ and $T_6$ into two adjacent triangular regions. We can then use Corollary \ref{prop_adj-triangles} to conclude that there are ordinary points, one in $T_1$ and one in $T_6$, different from the vertices $v_0$ and $v_1$. If, instead, the additional line through $v_1$ passes through $T_3$ and $T_5$ then, by a symmetric argument, we would be able to find ordinary points in each of these triangular regions different from $v_0$ and $v_1$. In either case, together with $v_2$, this would give us our three ordinary points. Hence, the only way for there to be \textit{fewer} than three ordinary points is for $v_1$ to be ordinary. But then we can apply precisely the same argument to the vertex $v_3$ to conclude that it too must be ordinary. Therefore, in that case, $v_1, v_2$, and $v_3$ are our three ordinary points. 
\end{proof}

With this proposition in hand, we can now easily prove our main result:

\begin{theorem}\label{main-thm}
Given a finite arrangement of lines in $\mathbb{RP}^2$, not all passing through a common point, there must be at least three ordinary points.
\end {theorem}

\begin{proof}
Either there are three lines in the arrangement whose pairwise intersection points are ordinary, yielding three ordinary points, or by Proposition \ref{prop_special_ord_point}, there is a triangular region, $T$, with a fourth line passing through a defining vertex of $T$ into its interior and intersecting the opposite edge of $T$ at an ordinary point. Applying Proposition \ref{prop-final} in this case gives us our three ordinary points.
\end{proof}



\section{Conclusions.}\label{sec:conclusions}


We provided a very simple and constructive proof that in an arrangement of $n$ lines, not all passing through a common point, there must be at least three ordinary points. While not constructive, equation \ref{melch-y-new}, which comes out of Melchior's argument, gives rich information that the constructive proof does not give, namely that the number of ordinary points can be precisely captured if one has a detailed knowledge of the faces with more than three sides, and the vertices with more than three lines passing through them. These ``excesses'' are what give rise to all ordinary points beyond the magic number $3$.


We remark that our result immediately generalizes to finite arrangements of pseudolines. An \textit{arrangement of pseudolines} is analogous to an arrangement of lines, but instead of a set of lines, one has a set of simple, noncontractible closed curves (again in the real projective plane) such that every pair of them cross each other transversely in a single point (as long as these pseudolines are not all concurrent). Note that the argument for three ordinary points, including Propositions \ref{prop_special_ord_point} and \ref{prop-final}, and Corollary \ref{prop_adj-triangles}, use no fact about finite arrangements of lines other than that they are collections of closed curves and that every pair of them intersect in a single point. Therefore, Theorem \ref{main-thm} is true for finite arrangements of pseudolines as well. For more on pseudoline arrangements, see \cite{grunbaum72}.

Despite the groundbreaking result of Green and Tao \cite{gt13} establishing an $\frac{n}{2}$ lower bound on the number of ordinary points as a function of the number of lines for large enough $n$, an old conjecture attributed to Gabriel Dirac and Theodore Motzkin \cite{dirac51, motzkin51} that there can be no fewer than $\lfloor \frac{n}{2} \rfloor$ ordinary points for all $n$ still stands.  
Perhaps adaptations of Propositions \ref{prop_special_ord_point} and \ref{prop-final}, and Corollary \ref{prop_adj-triangles} enabling one to actually count ordinary points could be used to help make inroads into this conjecture.

\bibliographystyle{abbrvnat}
\bibliography{main}

\end{document}